\documentclass[12pt,english,BCOR7.5mm]{amsart}
\usepackage{ae,aecompl}
\usepackage[T1]{fontenc}
\usepackage[latin9]{inputenc}
\usepackage[letterpaper]{geometry}
\usepackage{color}
\usepackage{babel}
\usepackage{amsthm}
\usepackage{amssymb}
\usepackage[unicode=true,
 bookmarks=true,bookmarksnumbered=true,bookmarksopen=true,bookmarksopenlevel=2,
 breaklinks=false,pdfborder={0 0 1},backref=false,colorlinks=true]
 {hyperref}
\hypersetup{pdftitle={Using XY-pc in LyX},
 pdfauthor={H. Peter Gumm},
 pdfsubject={LyX's XY-pic manual},
 pdfkeywords={LyX, documentation},
 linkcolor=black, citecolor=black, urlcolor=blue, filecolor=blue,pdfpagelayout=OneColumn, pdfnewwindow=true,pdfstartview=XYZ, plainpages=false, pdfpagelabels}

\makeatletter
\numberwithin{equation}{section}
\numberwithin{figure}{section}
  \theoremstyle{definition}
  \newtheorem*{problem*}{\protect\problemname}
\theoremstyle{plain}
\newtheorem{thm}{\protect\theoremname}
  \theoremstyle{plain}
  \newtheorem{lem}[thm]{\protect\lemmaname}
  \theoremstyle{plain}
  \newtheorem{cor}[thm]{\protect\corollaryname}
  \theoremstyle{definition}
  \newtheorem{defn}[thm]{\protect\definitionname}
  \theoremstyle{plain}
  \newtheorem{prop}[thm]{\protect\propositionname}
  \theoremstyle{plain}
  \newtheorem{conjecture}[thm]{\protect\conjecturename}

\usepackage[all]{xy}

\newcommand{\xyR}[1]{
  \xydef@\xymatrixrowsep@{#1}}
\newcommand{\xyC}[1]{
  \xydef@\xymatrixcolsep@{#1}}

\newdir{|>}{!/4.5pt/@{|}*:(1,-.2)@^{>}*:(1,+.2)@_{>}}

\let\myTOC\tableofcontents
\renewcommand\tableofcontents{%
  \pdfbookmark[1]{\contentsname}{}
  \myTOC }

\def\LyX{\texorpdfstring{%
  L\kern-.1667em\lower.25em\hbox{Y}\kern-.125emX\@}
  {LyX}}

\makeatother

  \providecommand{\conjecturename}{Conjecture}
  \providecommand{\corollaryname}{Corollary}
  \providecommand{\definitionname}{Definition}
  \providecommand{\lemmaname}{Lemma}
  \providecommand{\problemname}{Problem}
  \providecommand{\propositionname}{Proposition}
\providecommand{\theoremname}{Theorem}

\begin{document}

\title{Geometric Progression-Free Sequences with Small Gaps}

\author{Xiaoyu He}

\email{xiaoyuhe@college.harvard.edu}

\address{Eliot House, Harvard College, Cambridge, MA 02138.}

\date{\today}
\begin{abstract}
Various authors, including McNew, Nathanson and O'Bryant, have recently
studied the maximal asymptotic density of a geometric progression
free sequence of positive integers. In this paper we prove the existence
of geometric progression free sequences with small gaps, partially
answering a question posed originally by Beiglböck et al. Using probabilistic
methods we prove the existence of a sequence $T$ not containing any
$6$-term geometric progressions such that for any $x\geq1$ and $\varepsilon>0$
the interval $[x,x+C_{\varepsilon}\exp((C+\varepsilon)\log x/\log\log x)]$
contains an element of $T$, where $C=\frac{5}{6}\log2$ and $C_{\varepsilon}>0$
is a constant depending on $\varepsilon$. As an intermediate result
we prove a bound on sums of functions of the form $f(n)=\exp(-d_{k}(n))$
in very short intervals, where $d_{k}(n)$ is the number of positive
$k$-th powers dividing $n$, using methods similar to those that
Filaseta and Trifonov used to prove bounds on the gaps between $k$-th
power free integers.
\end{abstract}
\maketitle

\section*{Introduction}

Let $k\ge3$ be an integer, and $r$ be a positive rational number.
A geometric progression of length $k$ with common ratio $r$ is a
sequence $(a_{0},\ldots,a_{k-1})$ of nonzero real numbers for which
\[
a_{i}=ra_{i-1},i=1,\ldots,k-1.
\]
A $k$-geometric progression, or $k$-GP, is a geometric progression
of length $k$. Such a geometric progression is called \emph{trivial}
if $r=1$ and henceforth we consider only nontrivial progressions.

Rankin \cite{Rankin} introduced the notion of a $k$-GP-free sequence,
which for our purposes is a sequence of positive integers that contains
no nontrivial $k$-geometric progressions. Whereas a theorem of Szemerédi
\cite{Szemeredi} shows that $k$-term arithmetic-progression free
sequences must have zero upper density in the naturals, there exist
$3$-GP-free sequences with positive asymptotic density. For instance,
the sequence of all squarefree positive integers avoids all geometric
progressions of length $3$ or more, and has asymptotic density $\zeta(2)^{-1}=\frac{6}{\pi^{2}}$.

Let $A$ be a $k$-GP-free sequence. The question of finding the maximal
possible asymptotic density $d(A)$, or else the upper density $d_{U}(A)$,
of $A$ has been a subject of recent study \cite{BBHS,BG,McNew,NO,Riddell}.
For an exposition of progress on this problem and the tightest known
bounds on $d_{U}(A)$, along with constructions of $A$ with nearly
optimal upper density, see the paper of McNew \cite{McNew}.

In this paper we are interested in a uniform version of this density
problem. In particular, we would like to settle the existence of $k$-GP-free
sequences with small gaps. Beiglböck, Bergelson, Hindman, and Strauss
\cite{BBHS} proposed the following problem, which has implications
in ergodic theory.
\begin{problem*}
Does there exist a $c>1$, a $k\geq3$, and an increasing $k$-GP-free
sequence $T=\{t_{i}\}_{i\in\mathbb{N}}$ of positive integers such
that $t_{i+1}-t_{i}\leq c$ for all $i\in\mathbb{N}$? Such a sequence
with bounded gaps is called \emph{syndetic.}
\end{problem*}
We use the standard notations $f(x)=O(g(x))$ if there exists a constant
$A>0$ for which $f(x)<Ag(x)$ for all $x$, and $f(x)=o(g(x))$ if
for any constant $A>0$ the inequality $f(x)<Ag(x)$ holds for $x$
sufficiently large. Furthermore we write $f(x)=\Omega(g(x))$ if there
exists a constant $B>0$ for which $f(x)>Bg(x)$ for all $x$. In
both cases the range of $x$ depends on context but is usually the
natural numbers. We will often also use the shorthand $f(x)\ll g(x)$
for $f(x)=O(g(x))$, and similarly $f(x)\gg g(x)$ for $f(x)=\Omega(g(x))$. 

Note that the sequence of squarefree integers fails to be syndetic,
since there exist gaps as large as $\Omega(\log x/\log\log x)$ within
the squarefree integers up to $x$. Erd\H{o}s \cite{Erdos} knew this
elementary result in 1951, though he was not the first. He showed
that if $\{s_{i}\}_{i\in\mathbb{N}}$ is the sequence of squarefree
integers in increasing order, then
\[
s_{i+1}-s_{i}>(1+o(1))\frac{\pi^{3}}{6}\frac{\log s_{i}}{\log\log s_{i}}
\]
for infinitely many values of $i$. By roughly the same method, Rankin
\cite{Rankin2} proved that if $\{p_{i}\}_{i\in\mathbb{N}}$ is the
sequence of primes in increasing order, then 
\[
p_{i+1}-p_{i}>(1+o(1))e^{\gamma}\frac{\log p_{i}\log\log p_{i}\log\log\log\log p_{i}}{(\log\log\log p_{i})^{2}}
\]
for infinitely many values of $i$, where $\gamma$ is the Euler-Mascheroni
constant. The constant $e^{\gamma}$ in Rankin's bound was improved
several times; recently, it was replaced by a function growing to
infinity by Ford, Green, Konyagin and Tao \cite{FGKT} concurrently
with Maynard \cite{Maynard}, settling a \$10,000 Erd\H{o}s problem.
Most recently, Ford, Green, Konyagin, Maynard, and Tao combined the
previous methods and showed
\[
p_{i+1}-p_{i}\gg\frac{\log p_{i}\log\log p_{i}\log\log\log\log p_{i}}{\log\log\log p_{i}}
\]
for infinitely many $i$ and an effective implied constant, removing
a $\log\log\log p_{i}$ factor from the denominator. For a history
of these results and the current progress, see their paper \cite{FGKMT}.

It is a conjecture of Cramér \cite{Cramer} that
\[
\limsup_{i\rightarrow\infty}\frac{p_{i+1}-p_{i}}{(\log p_{i})^{2}}=1,
\]
but the best upper bound available is $p_{i+1}-p_{i}=O(p_{i}^{0.525})$
due to Baker, Harman, and Pintz \cite{BHP}. For a discussion of Cramér's
model and its deficiencies see the paper of Pintz \cite{Pintz}.

Large gaps between the squarefree integers up to $x$ are also poorly
understood. In this direction the tightest bound is due to Filaseta
and Trifonov \cite{FT}, that the largest gaps between squarefree
integers are at most $O(x^{\frac{1}{5}}\log x)$. Trifonov \cite{Trifonov}
established the generalization that the largest gaps between the $k$-th
power free integers up to $x$ are at most $O(x^{1/(2k+1)}\log x)$,
and assuming the \emph{$abc$} conjecture, Granville \cite{Granville}
was able to proved that the gaps are $O(x^{\varepsilon})$ for every
$\varepsilon>0$.

We are naturally interested in an unconditional construction of $k$-GP-free
sequences with gaps of size $O(x^{\varepsilon})$ for every $\varepsilon>0$.
Our main theorem proves the existence of such sequences using the
probabilistic method.
\begin{thm}
\label{thm:main}There exists a $6$-GP-free sequence $T$ of positive
integers $\{t_{i}\}_{i\in\mathbb{N}}$ such that for every $\varepsilon>0$,
there exists a constant $C_{\varepsilon}>0$ for which

\[
t_{i+1}-t_{i}<C_{\varepsilon}\exp\Big(\Big(\frac{5}{6}\log2+\varepsilon\Big)\frac{\log t_{i}}{\log\log t_{i}}\Big),
\]
holds for all $i\in\mathbb{N}$.
\end{thm}
Note that the previously best known unconditional result in this direction
is the bound by Trifonov on the sequence of $5$-th power free positive
integers, with gaps of size at most $O(x^{1/11}\log x)$. 

Let $d_{i}(n)$ denote the number of $i$-th powers dividing $n$,
and $d_{i,j}(n)$ be the number of pairs $(a,b)\in\mathbb{N}^{2}$
with $a^{i}b^{j}|n$. In the next section we prove a bound on short
sums of the form
\[
S_{i,j}(x,h,D)=\sum_{x<n\leq x+h}\exp(-Dd_{i,j}(n)),
\]
where $h$ is on the order of $\exp\Big((\frac{5}{6}\log2+\varepsilon)\frac{\log x}{\log\log x}\Big)$,
and $D$ is some constant. Once this computation is complete we will
use the probabilistic method to show that a sequence randomly generated
by removing terms from every $6$-GP has gaps of the desired size
with nonzero probability.

\section*{Sums of Divisor Functions in Short Intervals}

We will be interested on a lower bound on the quantity
\[
S_{i,j}(x,h,D)=\sum_{x<n\leq x+h}\exp(-Dd_{i,j}(n)),
\]
where $h$ is small and $D$ is some fixed positive constant.

The following uniform bound was proven by Nair and Tenenbaum \cite{NT}.
It is a generalization of a result of Shiu \cite{Shiu} to a larger
class of functions. We state the relevant special case of their theorem
here. Let $\mathcal{M}$ be the class of arithmetic functions $f:\mathbb{N}\rightarrow\mathbb{N}$
satisfying
\begin{enumerate}
\item (Submultiplicativity) The function $f$ satisfies $f(mn)\leq f(m)f(n)$
for $(m,n)=1$, and $f(1)=1$.
\item There exists a fixed $A>0$ such that for any prime $p,$ $f(p^{n})\leq A^{n}$
for all $n\in\mathbb{N}$.
\item For all $\varepsilon>0$, $f(n)=O(n^{\epsilon})$.\end{enumerate}
\begin{thm}
\label{thm:shiu}\cite{NT} Suppose that $f$ is a function in the
class $\mathcal{M}$. For any $0<\beta<\frac{1}{2}$ and a function
$h(x)\gg x^{\beta}$, we have
\[
\sum_{x<n\leq x+h(x)}f(n)\ll\frac{h(x)}{\log x}\exp\Big(\sum_{p\leq x}\frac{f(p)}{p}\Big),
\]
where the implicit constant depends only on $\beta$.
\end{thm}
We first prove that Theorem \ref{thm:shiu} applies to the functions
$d_{i,j}$ we are interested in.
\begin{lem}
For any $i,j\in\mathbb{N}$, function $d_{i,j}(n)$ which counts the
number of pairs $(a,b)\in\mathbb{N}^{2}$ for which $a^{i}b^{j}|n$
lies in the class $\mathcal{M}$.\end{lem}
\begin{proof}
Write
\[
D_{i,j}(n)=\{(a,b)\in\mathbb{N}^{2}:a^{i}b^{j}|n\},
\]
so that $d_{i,j}(n)=|D_{i,j}(n)|$. For the submultiplicativity of
$d_{i,j}$, it suffices to exhibit an injection $\phi:D_{i,j}(mn)\rightarrow D_{i,j}(m)\times D_{i,j}(n)$,
given $(m,n)=1$. Given $(a,b)\in D_{i,j}(mn)$, we let
\[
\phi((a,b))=((\mbox{gcd}(a,m),\mbox{gcd}(b,m)),(\mbox{gcd}(a,n),\mbox{gcd}(b,n))),
\]
and submultiplicativity follows. 

Next, for any prime $p$ and positive integer $n$, it is easy to
see
\[
d_{i,j}(p^{n})\leq(n+1)^{2},
\]
since there are at most $n+1$ choices for $a$ or $b$. But $(n+1)^{2}\leq4^{n}$
for all $n\in\mathbb{N}$, so condition 2 is satisfied. 

To show condition 3, we appeal to the well-known bound $d(n)=O(n^{\varepsilon})$,
where $d(n)$ is the divisor function. But $d_{i,j}(n)\leq d(n)^{2}$
since $a,b$ are both chosen from the divisors of $n$, and the growth
condition $d_{i,j}(n)=O(n^{\varepsilon})$ for all $\varepsilon>0$
holds.
\end{proof}
With Theorem \ref{thm:shiu} in hand, it is easy to show the following
bound.
\begin{cor}
\label{cor:shiu}For $i,j\geq2$, we have
\[
S_{i,j}(x,h,D)\gg h(x),
\]
if $h(x)\gg x^{\beta}$ for some $0<\beta<\frac{1}{2}$, where the
implicit constant depends only on $\beta$ and $D$.\end{cor}
\begin{proof}
We use Jensen's inequality and the convexity of the exponential function
to get
\[
S_{i,j}(x,h,D)\geq h(x)\exp\Big(-Dh(x)^{-1}\sum_{x<n\leq x+h}d_{i,j}(n)\Big).
\]

Applying Theorem \ref{thm:shiu}, we have then that
\[
S_{i,j}(x,h,D)\geq h(x)\exp\Big(-D(\log x)^{-1}\exp\Big(\sum_{p\leq x}\frac{d_{i,j}(p)}{p}\Big)\Big).
\]

Now for any prime $p$, $d_{i,j}(p)=1$, and so the inner sum is controlled
by Mertens' estimate \cite{Mertens}
\[
\sum_{p\leq x}\frac{1}{p}=\log\log x+O(1),
\]
from which the stated inequality follows.
\end{proof}
Corollary \ref{cor:shiu} will already suffice to give gaps of size
at most $O(x^{\varepsilon})$ for every $\varepsilon>0$. However,
we provide the following improvement to shorter intervals, using a
basic idea of Filaseta and Trifonov \cite{FT}. To the best of our
knowledge this bound, and the study of such sums in intervals shorter
than $x^{\varepsilon}$ for all $\varepsilon>0$ is new.
\begin{lem}
\label{lem:short}For any $\varepsilon>0$ and $i,j\geq2$, there
$E,N>0$ such that if $x>N$ then 
\[
S_{i,j}(x,h,D)\geq E\exp\Big((C_{i,j}+\varepsilon)\frac{\log x}{\log\log x}\Big),
\]
where $h(x)=\exp\Big((C_{i,j}+2\varepsilon)\frac{\log x}{\log\log x}\Big)$,
$C_{i,j}=\log2\Big(\frac{1}{i}+\frac{1}{j}\Big)$, and the constants
$E$ and $N$ depend only on the choices of $\varepsilon$ and $D$.\end{lem}
\begin{proof}
First, we have that $d_{i,j}(n)\leq d_{i}(n)d_{j}(n)$. Without loss
of generality assume $i\leq j$. Henceforth $p$ always refers to
a prime.

We enumerate the set $A$ of $n\in(x,x+h]$ divisible by some $i$-th
prime power $p^{i}$ where $p\leq h$. Because $i\geq2$, we have
that
\[
|A|\leq\sum_{p\leq h}\Big(\frac{h}{p^{i}}+1\Big)\leq(\zeta(i)-1)h+o(h),
\]
by the Chebyshev bound on the prime counting function, where $\zeta$
is the Riemann zeta function. Let $A^{c}$ denote the complement of
$A$ in $(x,x+h]$. It follows that since $\zeta(i)<2$ uniformly
in $i>2$, there exists a positive constant $0<B<1$ for which, whenever
$x$ is sufficiently large, $|A^{c}|\geq Bh$. We restrict our attention
to only $n\in A^{c}$.

Each such $n$ we can write as $n=p_{1}^{\alpha_{1}}p_{2}^{\alpha_{2}}\cdots p_{\ell}^{\alpha_{\ell}}m$
where each $p_{k}>h$, each $\alpha_{k}\geq i$, and $m$ is $i$-th
power free, so that 
\[
d_{i}(n)=\prod_{k=1}^{\ell}\Big(\Big\lfloor\frac{\alpha_{k}}{i}\Big\rfloor+1\Big).
\]

By a simple smoothing argument and the fact that if $n\in A^{c}$,
\[
\sum_{k=1}^{\ell}\alpha_{k}\leq\frac{\log n}{\log h}=\frac{(1+o(1))\log\log x}{C_{i,j}+2\varepsilon},
\]
we find that for any $n\in A^{c}$, and $x$ sufficiently large,
\[
d_{i}(n)\leq\exp\Big(\frac{\log2}{i}\frac{\log\log x}{C_{i,j}+\varepsilon}\Big),
\]
and similarly,
\[
d_{j}(n)\leq\exp\Big(\frac{\log2}{j}\frac{\log\log x}{C_{i,j}+\varepsilon}\Big).
\]

Combining these two inequalities, we find
\begin{eqnarray*}
S_{i,j}(x,h,D) & \geq & \sum_{n\in A^{c}}\exp\Big(-Dd_{i}(n)d_{j}(n)\Big)\\
 & \geq & Eh\exp\Big(-D(\log x)^{F}\Big),
\end{eqnarray*}
where 
\[
F=\frac{\log2}{C_{i,j}+\varepsilon}\Big(\frac{1}{i}+\frac{1}{j}\Big)<1.
\]
Plugging in the expression for $h$ the result follows, since
\[
\frac{(C_{i,j}+2\varepsilon)\log x}{\log\log x}-D(\log x)^{F}\geq\frac{(C_{i,j}+\varepsilon)\log x}{\log\log x}
\]
for $x$ sufficiently large in terms of $D$ and $\varepsilon$, since
$F<1$.
\end{proof}

\section*{The GP-Free Process}

For the proof of Theorem \ref{thm:main}, we randomly generate a $6$-GP-free
sequence by removing at least one element from each $6$-GP. For the
technical details to work out, we only remove one of the middle two
elements of each $6$-GP.

Let $G_{k}$ denote the family of all nontrivial $k$-term geometric
progressions of positive integers. Since each such family is countable,
we can enumerate $G_{k}$
\[
G{}_{k}=(G_{k,1},G_{k,2},\ldots),
\]
where each $G_{k,i}$ is a nontrivial $k$-GP and they are ordered
lexicographically as $k$-tuples of positive integers. To avoid double-counting
we assume that each $G_{k,i}$ has a common ratio $r_{k,i}>1$.
\begin{defn}
\label{6-gp}The \emph{$6$-GP-free process} randomly generates a
$6$-GP-free sequence $T$ as follows. Generate a sequence $U=(u_{1},u_{2},\ldots)$
where $u_{i}\in G_{6,i}$ such that if
\[
G_{6,i}=(a_{i}b_{i}^{5},a_{i}b_{i}^{4}c_{i},a_{i}b_{i}^{3}c_{i}^{2},a_{i}b_{i}^{2}c_{i}^{3},a_{i}b_{i}c_{i}^{4},a_{i}c_{i}^{5}),
\]
where $b_{i}<c_{i}$, then $u_{i}=a_{i}b_{i}^{3}c_{i}^{2}$ or $u_{i}=a_{i}b_{i}^{2}c_{i}^{3}$
with equal probability $\frac{1}{2}$. All of the $u_{i}$'s are chosen
independently. Let $T$ be the random variable whose value is the
sequence of all positive integers never appearing in $U$, sorted
in increasing order.
\end{defn}
It is clear that $T$ misses at least one term in each progression
in $G_{6}$. In the next section we show that with nonzero probability
$T$ has gaps smaller than $O\Big(\exp\Big((C_{2,3}+\varepsilon)\frac{\log x}{\log\log x}\Big)\Big)$
for every $\varepsilon>0$, thereby proving Theorem \ref{thm:main}.
We rely extensively on Lemma \ref{lem:short}.

\section*{Proof of the Main Theorem}

Using Definition \ref{6-gp} it suffices to prove the following lemma
to prove Theorem \ref{thm:main}.
\begin{lem}
\label{lem:main}For any $\varepsilon>0$, there exists a constant
$E>0$ such that for every constant $C>0$ and every $x\in\mathbb{N}$
sufficiently large, the probability that the sequence $T$ generated
from the $6$-GP-free process does not contain any element of $(x,x+Ch]$
is at most\textup{
\[
\mbox{P}[T\cap(x,x+Ch]=\emptyset]\leq\exp\Big(-CE\exp\Big((C_{2,3}+\varepsilon)\frac{\log x}{\log\log x}\Big)\Big),
\]
}\textup{\emph{where $C_{i,j}$, $h=h(x)$, and $E$ are as in Lemma
\ref{lem:short}, for the same $\varepsilon$ and $D=\log2$.}}\end{lem}
\begin{proof}
Fix $C>0$. The sequence $T$ contains no element of $(x,x+Ch]$ exactly
when the sequence $U$ in Definition \ref{6-gp} contains every element
of $(x,x+Ch]$. For any given positive integer $n\in\mathbb{N}$,
the probability that $n\in U$ can be controlled as follows. It is
easy to see that $n$ appears as the term $a_{i}b_{i}^{3}c_{i}^{2}$
or $a_{i}b_{i}^{2}c_{i}^{3}$ in some $G_{6,i}$ for at most $d_{3,2}(n)$
choices of $i$.

Furthermore, since the choices of $u_{i}\in G_{6,i}$ are independent,
it follows that
\[
\mbox{P}[n\not\in U]\geq\Big(\frac{1}{2}\Big)^{d_{3,2}(n)}.
\]

We can now bound the probability that every $ $$n\in(x,x+Ch]$ appears
in $U$. Given any fixed 
\[
G_{6,i}=(a_{i}b_{i}^{5},a_{i}b_{i}^{4}c_{i},a_{i}b_{i}^{3}c_{i}^{2},a_{i}b_{i}^{2}c_{i}^{3},a_{i}b_{i}c_{i}^{4},a_{i}c_{i}^{5}),
\]
note that the middle two elements $a_{i}b_{i}^{3}c_{i}^{2},a_{i}b_{i}^{2}c_{i}^{3}$
differ by at least $a_{i}b_{i}^{2}c_{i}^{2}$. Thus, if $n=a_{i}b_{i}^{3}c_{i}^{2}$
lies in $(x,x+Ch]$, then 
\[
|a_{i}b_{i}^{2}c_{i}^{3}-n|\geq a_{i}b_{i}^{2}c_{i}^{2}\geq\sqrt{n},
\]
and the same holds if $n=a_{i}b_{i}^{2}c_{i}^{3}$. For $x$ sufficiently
large we have $h(x)<\sqrt{x}$, so no two elements in $(x,x+Ch]$
lie in the same $G_{6,i}$ together. Thus each of the events $n\not\in U$
are independent. It follows that we can multiply probabilities, getting
\[
\mbox{P}[T\cap(x,x+Ch]=\emptyset]\leq\prod_{n\in(x,x+Ch]}\Big(1-\exp(-d_{3,2}(n)\log2)\Big),
\]
to which we can apply $1-t\leq e^{-t}$ to find that
\[
\mbox{P}[T\cap(x,x+Ch]=\emptyset]\leq\exp\Big(-\sum_{(x,x+Ch]}\exp(-d_{3,2}(n)\log2)\Big).
\]

Now the inner sum can be bounded by Lemma \ref{lem:short} with $D=\log2$,
from which we get the desired inequality
\[
\mbox{P}[T\cap(x,x+Ch]=\emptyset]\leq\exp\Big(-CE\exp\Big((C_{2,3}+\varepsilon)\frac{\log x}{\log\log x}\Big)\Big)
\]
for some constant $E>0$ depending only on the choice of $\varepsilon$,
for sufficiently large $x$.
\end{proof}
The proof of Theorem \ref{thm:main} is now dependent on choosing
$C$ large enough in Lemma \ref{lem:main}.
\begin{proof}
(Theorem \ref{thm:main}) For any fixed $\varepsilon>0$, we can pick
$C_{1}>0$ such that
\[
\sum_{x\in\mathbb{N}}\exp\Big(-C_{1}E\exp\Big((C_{2,3}+\varepsilon)\frac{\log x}{\log\log x}\Big)\Big)<1.
\]

In particular, by Lemma \ref{lem:main} and linearity of expectation,
we see that there exists an $N>0$ for which the expected number of
intervals $(x,x+C_{1}h]$ with $x>N$ that fail to intersect $T$
is less than 1. Thus, there exists a $T$ which intersects every $(x,x+C_{1}h]$
for all $x>N$. Taking $C$ sufficiently large so that for every $x\in\mathbb{N}$,
$(x,x+Ch(x)]$ contains an $(x',x'+C_{1}h(x')]$ for some $x'>N$,
we see that this $T$ intersects $(x,x+Ch]$ for all $x\in\mathbb{N}$.
Thus no gap can be of size greater than $Ch(x),$ as desired. The
value of $C_{2,3}$ is exactly $\frac{5}{6}\log2$, as claimed.
\end{proof}

\section*{Shorter Geometric Progressions}

In this section we study the case when $k<6$, and construct probabilistically
a sequence $T$ with gaps $O(x^{\varepsilon})$ which avoids $5$-term
geometric progressions. We also find such a $T$ avoiding $3$-term
geometric progressions with integer ratios. The proofs are very similar
to that of Theorem \ref{thm:main}, the only difference being we can
now apply Theorem \ref{thm:shiu} in place of our Lemma \ref{lem:short}.
\begin{prop}
\label{5-gp}There exists a $5$-GP-free sequence $T$ of positive
integers $\{t_{i}\}_{i\in\mathbb{N}}$ such that for every $\varepsilon>0$,
there exists a constant $C_{\varepsilon}>0$ satisfying

\[
t_{i+1}-t_{i}<C_{\varepsilon}t_{i}^{\varepsilon}
\]
for all $i\in\mathbb{N}$.\end{prop}
\begin{proof}
Enumerate the nontrivial $5$-GP's of positive integers, and let the
$n$-th one be
\[
(a_{n}b_{n}^{4},a_{n}b_{n}^{3}c_{n},a_{n}b_{n}^{2}c_{n}^{2},a_{n}b_{n}c_{n}^{3},a_{n}c_{n}^{4})
\]
with ratio $c_{n}/b_{n}>1$. We construct $T$ by removing from $\mathbb{N}$,
randomly and independently, one of the two terms $a_{n}b_{n}^{3}c_{n}$
and $a_{n}b_{n}^{2}c_{n}^{2}$ from each such geometric progression,
with probabilities $1-p,p$ where $p=p(a_{n}b_{n}^{2}c_{n}^{2})$
depends only on the middle term of the progression, and is a nondecreasing
function $p:\mathbb{N}\rightarrow(0,1)$ we will choose to be 
\[
p(x)=1-\frac{1}{\log(x+2)}.
\]

For a given $x\in\mathbb{N}$, we compute the probability $P_{x}$
that $x$ is not removed in this process. By definition, $x$ is the
second term $a_{n}b_{n}^{3}c_{n}$ in at most $d_{3,1}(x)$ distinct
progressions, and the third term $a_{n}b_{n}^{2}c_{n}^{2}$ in at
most $d_{2,2}(x)$ distinct progressions. Because $c_{n}>b_{n}$,
we have
\[
a_{n}b_{n}^{2}c_{n}^{2}\geq a_{n}b_{n}^{3}c_{n},
\]
and so
\[
P_{x}\geq p(x)^{d_{3,1}(x)}(1-p(x))^{d_{2,2}(x)},
\]
by the independence of all of these events. To finish the proof, pick
$h(x)=Cx^{\varepsilon}$ to be the interval length. Then, by independence
we get
\[
\mbox{P}[T\cap(x,x+h(x)]=\emptyset]\leq\prod_{i=1}^{h(x)}\Big(1-p(x+i)^{d_{3,1}(x+i)}(1-p(x+i))^{d_{2,2}(x+i)}\Big),
\]
and applying $1-t\leq e^{-t}$,
\[
\mbox{P}[T\cap(x,x+h(x)]=\emptyset]\leq e^{-S},
\]
where
\[
S=\sum_{i=1}^{h(x)}p(x+i)^{d_{3,1}(x+i)}(1-p(x+i)^{d_{2,2}(x+i)}).
\]

Applying Jensen's inequality, we can further bound $S$ by
\begin{eqnarray*}
S & \geq & h\exp\Big(\frac{1}{h}\sum_{i=1}^{h(x)}d_{3,1}(x+i)\log p(x+i)+d_{2,2}(x+i)\log(1-p(x+i))\Big)\\
 & \gg & h\exp\Big(\log p(x)\log x+\log(1-p(x))\Big),
\end{eqnarray*}
where we applied Theorem \ref{thm:shiu} on the two sums
\[
\sum_{i=1}^{h(x)}d_{3,1}(x+i)\ll h(x)\log x,
\]
since $d_{3,1}(p)=2$ for all primes $p$, and
\[
\sum_{i=1}^{h(x)}d_{2,2}(x+i)\ll h(x)
\]
since $d_{2,2}(p)=1$ for all primes $p$. We also used the fact that
$p$ is essentially constant on intervals of length $h$. By choosing
$p(x)\sim1-\frac{1}{\log x}$ in the above, we find that
\[
\log p(x)\log x+\log(1-p(x))=O(\log\log x)
\]
and so we have shown $S\gg h(x)/(\log x)^{E}\gg x^{\varepsilon/2}$
for some constant $E$, whence we can choose $C$ sufficiently large
in $h(x)=Cx^{\varepsilon}$ such that the sum $\sum e^{-S}<1$, proving
the existence of a $T$ not missing any elements in intervals $(x,x+h(x)]$.
\end{proof}
If we further restrict to avoiding only geometric progressions with
integer ratios, then we have by a similar argument the same bound
for $k=3$.
\begin{prop}
There exists a sequence $T$ of positive integers $\{t_{i}\}_{i\in\mathbb{N}}$,
avoiding all $3$-GP's with integer ratios, such that for every $\varepsilon>0$,
there exists a constant $C_{\varepsilon}>0$,
\[
t_{i+1}-t_{i}<C_{\varepsilon}t_{i}^{\varepsilon}
\]
for all $i\in\mathbb{N}$.\end{prop}
\begin{proof}
Enumerate the integer-ratio $3$-GP's of positive integers, such that
the $n$-th one is
\[
(a_{n},a_{n}r_{n},a_{n}r_{n}^{2})
\]
with integer ratio $r_{n}\geq2$. Construct $T$ by removing $a_{n}r_{n}$
or $a_{n}r_{n}^{2}$ from each such progression, randomly and independently,
with probabilities $1-p,p$ respectively, where $p=p(a_{n}r_{n}^{2})$
is chosen as the nondecreasing function 
\[
p(x)=1-\frac{1}{\log(x+2)}.
\]

Then if $P_{x}$ is the probability a given element $x\in\mathbb{N}$
remains in $T$, it is bounded by
\[
P_{x}\geq p(x)^{d(x)}(1-p(x))^{d_{2}(x)}.
\]

Following the exact same computation as in the proof of Proposition
\ref{5-gp}, we get
\[
\mbox{P}[T\cap(x+h(x)]=\emptyset]\leq e^{-S},
\]
and by two applications of Theorem \ref{thm:shiu}, $S$ is bounded
by
\begin{eqnarray*}
S & \gg & h\exp\Big(\log p(x)\log x+\log(1-p(x))\Big)\\
 & \gg & h/\log(x)^{E}
\end{eqnarray*}
for some constant $E>0$. As $h(x)=Cx^{\varepsilon},$ this shows
that $S\gg x^{\varepsilon/2}$ whereby we can pick $C$ sufficiently
large that $\sum e^{-S}<1$, proving the existence of a $T$ intersecting
all $(x,x+h(x)]$.
\end{proof}
Unfortunately, substantial improvement via shortening the intervals
in the bound of Nair and Tenenbaum seems unlikely, due to the fact
that individual values of $d_{i,j}(n)$ can grow quite large. In particular,
it is a classical theorem of Wigert that
\[
\limsup_{n\rightarrow\infty}\frac{\log d(n)}{\log n/\log\log n}=\log2,
\]
and the upper behavior of general $d_{i,j}$ is similar. Proceeding
directly by convexity and the bound of Nair and Tenenbaum, one cannot
hope to do better than Theorem \ref{thm:main}, except possibly by
improving the constant.

\section*{Closing Remarks and Open Questions}

The large gaps problems for primes and squarefree integers remain
poorly understood. Whereas the largest gaps in both sequences are
expected to be $O((\log x)^{2})$, and certainly $O(x^{\varepsilon})$
for every $\varepsilon>0$, the best that has been shown is $O(x^{0.525})$
due to Baker, Harman, and Pintz \cite{BHP} for the primes and $O(x^{\frac{1}{5}}\log x)$
by Filaseta and Trifonov \cite{FT} for the squarefree integers. In
light of these difficulties, a solution of the short gaps problem
for GP-free sequences should sidestep our lack of understanding of
the distribution of primes and squarefree integers.

Using the Chinese Remainder Theorem, it is easy to construct runs
of consecutive integers of order approximately $\log x$ that avoid
squarefree integers (and thus primes as well). We expect that a natural
barrier of a similar type occurs for sequences generated by the $6$-GP-free
process, barring it from directly solving the original question of
Beiglböck, Bergelson, Hindman and Strauss on syndetic sequences. Nevertheless
improving Theorem \ref{thm:main} to the order of $\log x$ would
already be significant. We make the following conjecture.
\begin{conjecture}
There exists a constant $C>0$ such that with nonzero probability,
the sequence $T$ generated by the $6$-GP-free process has gaps of
size $O((\log x)^{C})$.
\end{conjecture}
It was originally a conjecture of Cramér \cite{Cramer} that the primes
behave on a large scale as if they are randomly chosen out of the
integers, with $n$ prime with probability $(\log n)^{-1}$. Of course
this agrees with the Prime Number Theorem, but for short intervals
of size $O((\log x)^{C})$ a number of significant discrepancies between
Cramér's predictions and the distribution of primes have been found
by Maier \cite{Maier} and Pintz \cite{Pintz}. Our method is based
on the intuition that a truly randomized sequence generated by the
$6$-GP-free process or a similar probabilistic construction avoids
these issues while still carrying the required multiplicative structure.

In the other direction, using a computer search we have found that
any sequence of positive integers with gaps of size at most one contains
$3$-GP's. In particular, an exhaustive search of all sequences containing
one of every pair of consecutive integers in $[1,640]$ showed that
all such sequences contain at least one $3$-GP. This is weak evidence
that the original problem of Beiglböck, Bergelson, Hindman and Strauss
about syndetic sequences is unlikely to be true, and in fact all syndetic
sequences must contain arbitrarily long $k$-term geometric progressions.

\section*{Acknowledgements}

This research was conducted as part of the University of Minnesota
Duluth REU program of 2014, supported by National Science Foundation
grant number DMS-1358659 and National Security Agency grant number
H98230-13-1-0273. Thanks go out to Joe Gallian for writing advice
and for organizing the REU program, as well as Kevin O'Bryant, Levent
Alpoge and Samuel Zbarsky for providing ideas in the early stages
of this research. Also, I am grateful to Noah Arbesfeld, Daniel Kriz
and Arul Shankar for useful comments on this paper. Also, the author
is grateful to the anonymous referee for writing suggestions.

\end{document}